\documentclass[reqno,11pt]{amsart}
\usepackage[utf8]{inputenc}
\usepackage{amsmath}
\usepackage{amsthm}
\usepackage{amssymb}
\usepackage{amsfonts,mathrsfs}
\usepackage{stmaryrd}
\usepackage{amsxtra}  
\usepackage{epsfig}
\usepackage{verbatim}
\usepackage{enumerate}
\usepackage{xcolor}
\usepackage[all]{xy}

\usepackage{mathtools}
\usepackage{tikz}
\usetikzlibrary{shapes}
\usetikzlibrary{plotmarks}
\usepackage{bm}
\allowdisplaybreaks

\theoremstyle{plain}
\newtheorem{theorem}[equation]{Theorem}
\newtheorem{proposition}[equation]{Proposition}
\newtheorem{lemma}[equation]{Lemma}
\newtheorem{corollary}[equation]{Corollary}
\newtheorem{definition}[equation]{Definition}

\theoremstyle{remark}
\theoremstyle{remark}
\newtheorem{remark}[equation]{Remark}
\numberwithin{equation}{section}

\newcommand{\abs}[1]{\left\vert#1\right\vert}
\newcommand{\norm}[1]{\left\Vert#1\right\Vert}

\newcommand{\xdownarrow}[1]{%
  {\left\downarrow\vbox to #1{}\right.\kern-\nulldelimiterspace}
}
\makeatletter
\newcommand*{\dashdownarrow}{%
  \mathrel{%
    \mathpalette\dasharrow@vert{-90}%
  }%
}
\newcommand*{\dashuparrow}{%
  \mathrel{%
    \mathpalette\dasharrow@vert{90}%
  }%
}
\newcommand*{\dasharrow@vert}[2]{%
  \sbox0{$#1\vcenter{}$}%
  \sbox2{$#1\dashrightarrow\m@th$}%
  \dimen@=1.2\dimexpr\ht2-\ht0\relax
  \sbox2{\raisebox{-\ht0}{\unhcopy2}}%
  \ht2=\z@
  \dp2=\z@
  \vcenter{\hbox to 2\dimen@{\hfill\rotatebox{#2}{\box2}\hfill}}%
}
\makeatother
\newcommand{\cf}{{\mathcal F}}

\newcommand{\co}{{\mathcal O}}

\newcommand{\cR}{{\mathcal R}}

\newcommand{\C}{{\mathbb C}}

\newcommand{\N}{{\mathbb N}}

\newcommand{\R}{{\mathbb R}}
\newcommand{\T}{{\mathbb T}}

\newcommand{\Z}{{\mathbb Z}}

\newcommand{\overbar}[1]{\mkern 1.5mu\overline{\mkern-1.5mu#1\mkern-1.5mu}\mkern 1.5mu}

\begin{document}

\title[Convergence in $A^1$]{Norm convergence of partial sums of $H^1$ functions}
\author{J. D. McNeal \& J. Xiong}
\begin{abstract}
A classical observation of Riesz says that truncations of a general $\sum_{n=0}^\infty a_n z^n$ in the Hardy space $H^1$ do not converge in $H^1$. 
A substitute positive result is proved: these partial sums always converge in the Bergman norm $A^1$. The result is extended to complete Reinhardt domains in $\C^n$.
A new proof of the failure of $H^1$ convergence is also given.
 \end{abstract}
%\date{\today}
\address{Department of Mathematics, \newline The Ohio State University, Columbus, Ohio, USA}
\email{mcneal@math.ohio-state.edu}
\address{Department of Mathematics, \newline The Ohio State University, Columbus, Ohio, USA}
\email{xiong.147@osu.edu}

\maketitle

%%%%%%%%%%%%%%%%%%%%%%%%%%%%%%%%%%%%%%%%%%%%%%%%%%%%%%%%%%%%%%%%%%%%%%%%%%%%%%%%%%%%%%%%%%%%%%%%%%%%%%%%%%%%%%%%%%%%%%%%%%%%%%%%%%%%%%%%%%%%%%%%% INTRODUCTION %%%%%%%%%%%%%%%%%%%%%%%%%%%%%%%%%%%%%%%%%%%%%%%%%%%%%%%%%%%%%%%%%%%%%%%%%%%%%%%%%%%%%%%%%%%%%%%%%%%%%%%%%%%%%%%%%%%%%%%%%%%%%%%%%%%%%%%%%%%%%%%%%%%%%%%%%%%%%%%%%%%%%%%%%%%%%%%%%%%%%%%%%%%%%%%%%%%%%%%%%%%%%%%%%%%%%%%%%%%%%%%%%%%%%%%%%%%%%%%%%%%%%%%%%%%%%

\section{Introduction}\label{S:intro}

Let $U\subset\C$ be the unit disc and $\co(U)$ denote the set of holomorphic functions on $U$. If $f\in\co(U)$, with power series $f(z)=\sum_{k=0}^\infty a_k z^k$, then
\begin{equation*}
S_Nf(z) := \sum_{k=0}^N a_k z^k \longrightarrow f(z)
\end{equation*}
uniformly on compact subsets of $U$. If $\big(X, \|\cdot\|_X\big)\subset\co(U)$ is a Banach space of functions, it is natural to ask whether $S_N f$ also converges 
to $f$ in the norm $\|\cdot\|_X$. 

Two classically studied  spaces, the Bergman and Hardy spaces, will be considered here. For $p>0$, the Bergman space $A^p(U)$ is the set of $f\in\co (U)$  such that
\begin{equation*}
\left\|f\right\|^p_{A^p(U)} = \int_U |f|^p\, dV <\infty,
\end{equation*}
$dV$ denoting Lebesgue measure on $\C$. The Hardy space $H^p(U)$ is the set of $f\in\co (U)$ such that
\begin{equation}\label{E:hardy}
\left\|f\right\|^p_{H^p(U)} =\sup_{0\leq r <1}\frac 1{2\pi} \int_{0}^{2\pi} \left|f(r e^{i\theta})\right|^p\, d\theta <\infty,
\end{equation}
$d\theta$ denoting Lebesgue measure on $[0, 2\pi]$. For $1\leq p<\infty$, $A^p(U)$ and $H^p(U)$ are Banach spaces. 

When $p=2$, norm convergence of $S_Nf$ in either the Bergman or Hardy norm is elementary. For $f\in H^2(U)$, orthogonality of $\left\{e^{ik\theta}\right\}$ on $\partial U$ shows $\|f\|^2_{H^2(U)} =\sum_{k=0}^\infty |a_k|^2$. Orthogonality also shows
$\left\|S_N f-f\right\|^2_{H^2(U)} =\sum_{k=N+1}^\infty\left| a_k\right|^2$,
which tends to $0$ as $N\to\infty$. Minor modifications of the argument hold when $A^2(U)$ replaces $H^2(U)$. When $1<p<\infty$ and $p\neq 2$, the result is the same as in the Hilbert space case but proving this 
is no longer elementary. Norm convergence of $S_N f$ in $H^p(U)$ for $1<p<\infty$ is considered classical; a proof is contained in \cite{Garnettbook} on pages 104--110. Convergence of $S_N f$ in $A^p(U)$ for the same range of $p$ is established by Zhu \cite{Zhu91}, utilizing the result on $H^p(U)$.

The focus in this paper is $p=1$. Our interest in this case stems from the widespread occurrence of $L^1$ holomorphic functions, not as an endpoint consideration.
For $A^1(U)$ and $H^1(U)$, it is known that partial sum approximation fails; this is also addressed in  \cite{Garnettbook} and \cite{Zhu91}. For $A^1(U)$, \cite{Zhu91} gives an explicit family of functions $g_\alpha\in A^1(U)$, $\alpha\in U$, such that $\left\|S_N g_\alpha\right\|_{A^1}$ is not bounded uniformly in $\alpha$ and $N$. The fact that $\left\|S_N f-f\right\|_{A^1}\not\to 0$ for all $f\in A^1(U)$ then follows from the uniform boundedness principle. 
For $H^1(U)$, the proofs in print are somewhat oblique. In \cite{Garnettbook},  it is first shown that $S_Nf$ approximating $f$ in $H^1$ is
equivalent to $h^1$ boundedness of the harmonic conjugation operator, where $h^1(U)$ denotes the harmonic functions with the norm $\|\cdot\|_{H^1}$. Both properties are also shown to be equivalent to the $L^1$ boundedness of the Szeg\H o projection on $U$.
The fact that harmonic conjugation is not bounded on $h^1(U)$ -- evidenced, e.g., by the Poisson kernel, as discussed in  \cite{Garnettbook} and \cite{Durenbook} --  then implies that partial sum approximation on $H^1(U)$ fails.

The role of harmonic conjugation in this argument does not readily generalize to domains in $\C^n$ or to non-simply connected domains in the plane. Consequently a new proof of the failure of $H^1$-approximation of partial sums,  in the spirit of \cite{Zhu91}, is given in the Section \ref{SS:failure}.

Using polar coordinates, it is easy to see $H^p(U)\subset A^p(U)$.  The main purpose of this paper is a substitute positive result for the failure of $H^1$ partial sum approximation: partial sums of $f\in H^1(U)$ are norm convergent, but in the weaker norm $A^1(U)$.

\begin{theorem}\label{T:main}
If $f(z)=\sum_{k=0}^\infty a_k z^k\in H^1(U)$ and $S_Nf(z) = \sum_{k=0}^N a_k z^k$, then

$$\left\|S_Nf-f\right\|_{A^1(U)}\longrightarrow 0\qquad\text{as } N\to\infty.$$

\end{theorem}

In the Section \ref{S:several}, Theorem \ref{T:main} is extended to give an analogous result on complete bounded Reinhardt domains $\cR\subset\C^n$.

There are other substitutes for the failure of $H^1(U)$ partial sum approximation in the literature: \cite{Sisk87,Sisk90,SiskAle95} on the $H^1(U)$ boundedness of the Ces\`aro operator, \cite{Stegenga76,PapVir08} on certain Toeplitz and Hankel operators,  \cite{Sisk01,Lif13} on boundedness of  the Hausdorff operator with particular choices of Borel measure, and 
\cite{PavNow10} on boundedness of the Libera operator from $H^1(U)$ to $H^p(U)$ with $0<p<1$. Unlike Theorem \ref{T:main}, these results involve modifications of $S_Nf$.

As notational shorthand, $\abs{a}\lesssim \abs{b}$ will mean there exists a constant $C>0$ such that $\abs{a}\leq C \abs{b}$, with $C$ independent of specified parameters.
Let $\abs{a}\approx \abs{b}$ mean both $\abs{a}\lesssim \abs{b}$ and  $\abs{b}\lesssim \abs{a}$ hold.

%%%%%%%%%%%%%%%%%%%%%%%%%%%% Partial sum of H^1(U) %%%%%%%%%%%%%%%%%%%%%%%%%%%%%%%%%%%
\section{The result on $U$}\label{S:one}

For $h(z)=\sum_{k=0}^\infty a_k z^k \in\co(U)$ and $N\in \Z^+$, let $S_Nh(z)=\sum_{k=0}^N a_k z^k$ denote the $N$-th partial sum of the power series of $h$.
Since power series are unique, call these polynomials {\it partial sums of} $h$ for short.

\subsection{Failure of norm convergence in $H^1(U)$}\label{SS:failure} A family of integral estimates is used in the proof of Theorem \ref{T:fail} below.
\begin{lemma}\label{calculus lemma}
For $z\in U$ and $c$ real, define
\[I_c(z)=\int_{0}^{2\pi}\frac{1}{\abs{1-ze^{-i\theta}}^{1+c}}\, d\theta. \]
If $c<0$ then $I_c \in L^\infty (U)$. Furthermore,  
\[I_c(z)\approx \frac{1}{(1-\abs{z}^2)^c} \text{ if } c>0\text{ and }\,\, I_{0}(z)\approx \log\frac{1}{1-\abs{z}^2}, \]
for constants independent of $z\in U$.
\end{lemma}

\begin{proof}
See \cite{RudinFunctiontheory} Proposition 1.4.10, \cite{ForRud74}, or \cite{ZhuBergmanbook} for the standard proof involving asymptotics of the Gamma function. See \cite{DurSch04, EdhMcN16, EdhMcN16b} for alternate, elementary proofs that extend to other singular integrands.
\end{proof}

\begin{theorem}\label{T:fail}
There exists $g\in H^1(U)$ such that $S_N g$ does not converge in $H^1(U)$.
\end{theorem}

\begin{proof}
For $a\in U$, define $f_a(z)=\frac{1-\abs{a}^2}{(1-\overbar{a}z)^2}$. By Lemma \ref{calculus lemma}, $\norm{f_a}_{H^1(U)}\lesssim 1$ with constant independent of $a$. The power series of $f_a$ is
$$f_a(z)=\sum_{k=0}^{\infty}(1-\abs{a}^2) (k+1)(\overbar{a}z)^k.$$
Consider the partial sum
\begin{equation*}
\begin{split}
S_{N}f_a(z)
&=\sum_{k=0}^{N}(1-\abs{a}^2) (k+1)(\overbar{a}z)^k = (1-\abs{a}^2)\sum_{k=0}^N \frac d{dt}\left(t^{k+1}\right)_{|_{t=\bar a z}}\\
&=(1-\abs{a}^2)\bigg[\frac{1-(\overbar{a}z)^{N+2}}{(1-\overbar{a}z)^2}-\frac{(N+2)(\overbar{a}z)^{N+1}}{1-\overbar{a}z}\bigg]\\
&=T_1+T_2.
\end{split}
\end{equation*}
$\|T_1\|_{H^1}$ is uniformly bounded in $N$ and $a$, since
\[
\frac{1}{2\pi}\int_{0}^{2\pi} \frac{(1-\abs{a}^2)\abs{1-(\overbar{a}e^{i\theta})^{N+2}}}{\abs{1-\overbar{a}e^{i\theta}}^2} d\theta
\leq \frac{2}{2\pi}\int_{0}^{2\pi} \frac{(1-\abs{a}^2)}{\abs{1-\overbar{a}e^{i\theta}}^2} d\theta
<\infty
\]
by Lemma \ref{calculus lemma}. $\|T_2\|_{H^1}$ is estimated
\[
\frac{(1-\abs{a}^2)}{2\pi}\int_{0}^{2\pi} \frac{(N+2)\abs{\overbar{a}e^{i\theta}}^{N+1}}{\abs{1-\overbar{a}e^{i\theta}}}d\theta
\gtrsim (1-\abs{a}^2)\abs{a}^{N+1}(N+2)\log \frac{1}{1-\abs{a}}
\]
for a constant independent of $N$ and $a$, by Lemma \ref{calculus lemma}. 
Let $a=\frac{N}{N+1}$. The lower bound on $\|T_2\|_{H^1}$ goes to infinity as $N\rightarrow +\infty$. Thus $\left\| S_N f_a\right\|_{H^1}$ is unbounded as a function of $N$ and $a$. The uniform boundedness principle in contrapositive form gives the stated conclusion.
\end{proof}

\begin{remark}  Holomorphic polynomials are dense in $H^1(U)$ (see Proposition \ref{P:polysDense}), but Theorem \ref{T:fail} says the sequence of natural polynomials 
$\left\{S_N h\right\}$ does not approximate a general $h\in H^1(U)$. Is there a best association of $h\in H^1(U)$ to a sequence of holomorphic polynomials $\left\{p_n\right\}$ such that $\|p_n-h\|_{H^1}\to 0$? Several interpretations of ``best'' are possible; the authors are unaware of results in this direction.
\end{remark}

%%%%%%%%%%%%%%%%%%%%%%%%%%%%%%%%%%%%%%%%%%%%%%%%%%%%%%%%%%%%%%%%%%%%%%%%%%%%
\subsection{Convergence in $A^1(U)$}\label{SS:proof1} Convergence of $S_Nf$ can be reduced to a bound on the operator norm of $S_N$. The following is a slight generalization of  \cite[Proposition 1]{Zhu91}.

\begin{lemma}\label{L:ZhuPlus}
Let $T_k$, $k=1,2, \dots$, be a sequence of bounded linear operators from a Banach space $X$ to a Banach space $Y$. Suppose that there is a dense subset 
$D$ of $X$ such that  for each $x\in D$,  $T_k x \to 0$ in the norm of $Y$ as $k\to \infty$. 

Then the following are equivalent
\begin{enumerate}
\item[(i)] $\lim_{k\to \infty} \norm{T_kx}_Y=0$ for each $x\in X$.
\item[(ii)] there is a $C>0$ such that for each $k$, we have $\norm{T_k}_{\rm op}\leq C$.
\end{enumerate}
$\norm{T_k}_{\rm op}$ is the operator norm of $T_k:X\to Y$.
\end{lemma}

\begin{proof}   Assume (i). Then (ii) holds by the uniform boundedness principle.

Assume (ii). Fix $x\in X$ and $\epsilon>0$. Since $D$ is dense in $X$, there exists $p\in D$ such that $\norm{x-p}_X< \frac{\epsilon}{2C}$. Therefore
\begin{align*}
\norm{T_k x}_Y&\leq \norm{T_k x - T_k p}_Y+ \norm{T_k p}_Y  < \frac{\epsilon}{2}+  \norm{T_k p}_Y.
 \end{align*} 
Choosing $k$ so large that $\norm{T_k p}_Y< \frac{\epsilon}{2}$ yields (i). 
\end{proof}

The key idea of the proof of Theorem \ref{T:main} is to represent the coefficients in $S_Nf$ as integrals, reducing the problem to estimates of geometric series.

\begin{proof}[Proof of Theorem \ref{T:main}]
For $b\in\overbar U$ and $\rho >0$, let $U(b; \rho)$ denote the disc centered at $b$ of radius $\rho$. Let $U_r =U(0; r)$. For each  $U_r$, choose $U_R$ such that $0<r<R<1$. If $f(z)=\sum_{j=0}^{\infty} a_j z^j\in H^1(U)$, the Cauchy integral formula gives
$a_j=\frac{1}{2\pi i}\int_{\partial U_R}\frac{f(\xi)}{\xi^{j+1}}d\xi$.
Therefore
\begin{align*}
\int_{U_r} \abs{\sum_{j=0}^N a_j z^j} dV
&=\int_{U_r} \abs{\sum_{j=0}^N\frac{1}{2\pi i}\int_{\partial U_R}\frac{f(\xi)}{\xi^{j+1}}d\xi z^j} dV(z)\\
&=\frac{1}{2\pi}\int_{U_r} \abs{\int_{\partial U_R}f(\xi)\sum_{j=0}^N \frac{1}{\xi}\bigg(\frac{z}{\xi}\bigg)^j d\xi} dV(z)\\
&=\frac{1}{2\pi}\int_{U_r} \abs{\int_{\partial U_R }f(\xi)\frac{1-\big(\frac{z}{\xi}\big)^{N+1}}{\xi-z} d\xi} dV(z)\\
&\lesssim\int_{U_r} \int_{\partial U_R} \abs{f(\xi)} \abs{\frac{1-\big(\frac{z}{\xi}\big)^{N+1}}{\xi-z}} d\abs{\xi} dV(z)=I.\\
\end{align*}
Since $\xi\in \partial U_R$ and $z\in U_r$, $\abs{\frac{z}{\xi}}<\frac{r}{R}<1$. Thus Fubini's theorem implies
\begin{align}\label{E:1}
I\lesssim \int_{U_r} \int_{\partial U_R} \abs{f(\xi)} \frac{1}{\abs{\xi-z}} d\abs{\xi} dV(z)=\int_{\partial U_R} \abs{f(\xi)} \int_{U_r}\frac{1}{\abs{\xi-z}} dV(z) d\abs{\xi},
\end{align}
with constant independent of $N$. 

For any $\xi\in \partial U_R$ fixed,  note $U\subset U(\xi; 2)$. Letting $z=\xi+s e^{i\sigma}$,
\begin{align*}\label{E:2}
\int_{U_r}\frac{1}{\abs{\xi-z}} dV(z)
< \int_{U_{(\xi; 2)}}\frac{1}{\abs{\xi-z}}dV(z)
=\int_0^{2\pi}\int_0^2\frac{1}{s} s\,ds\, d\sigma\lesssim 1,
\end{align*}
with constant independent of $\xi$. Thus \eqref{E:1} implies 
\begin{equation*}
\int_{U_r} \abs{\sum_{j=0}^N a_j z^j} dV
\lesssim \int_{\partial U_R} \abs{f(\xi)} d\abs{\xi}
\lesssim \norm{f}_{H^1(U)},
\end{equation*}
with constant independent of $N$.
Since $\lim_{r\rightarrow 1} \norm{S_N f}_{A^1(U_r)}=\norm{S_N f}_{A^1(U)}$, 
\begin{equation}\label{E:3}
\norm{S_N f}_{A^1(U)}\lesssim \norm{f}_{H^1(U)}.
\end{equation}

Let $T_N=S_N -\text{id}$, $X=H^1(U)$, $Y=A^1(U)$, and $D=\{\text{holomorphic polynomials}\}$ in Lemma \ref{L:ZhuPlus}. Note that for any $p\in D$, $T_N p\equiv 0$ if $N\geq \text{deg}(p)$. 
Lemma \ref{L:ZhuPlus} says that 
 \eqref{E:3} implies 
$\left\|S_Nf-f\right\|_{A^1(U)}\to 0$ as $N\to\infty$.
\end{proof}

%%%%%%%%%%%%%%%%%%%%%%%%%%%%%%%%%%%%%%%%%%%%%%%%%%%%%%%%%%%%%%%%%%%%%%%
\section{Several variable extension}\label{S:several}

 A domain $\Omega\subset\C^n$ is a {\it complete Reinhardt domain} if $(z_1,\dots z_n)\in\Omega$ implies $(\lambda_1 z_1,\dots , \lambda_n z_n)\in\Omega$ for
all $\lambda_k\in\C$ with $|\lambda_k|<1$, $k=1,\dots n$.

Let $\mathcal{R}$ be a bounded complete Reinhardt domain in $\C^n$ and $\mathcal{O}(\mathcal{R})$ denote the set of holomorphic functions on $\mathcal{R}$. Each $ f\in \mathcal{O}(\mathcal{R})$ has a power series expansion
$f(z)=\sum_{\alpha\in\N^n}a_{\alpha}z^{\alpha}$, using standard multi-index notation,
converging uniformly on compact subsets of $\mathcal{R}$.

A choice of partial sum of $f\in \mathcal{O}(\mathcal{R})$ is required, since the index set is an $n$-dimensional lattice. Let $\abs{\alpha}_{\infty}=\max\{\alpha_j: \alpha=(\alpha_1,\cdots,\alpha_n)\}$ if $\alpha\in\N^n$.  For $f(z)=\sum_{\alpha} b_{\alpha}z^{\alpha}\in \mathcal{O}(\mathcal{R})$, define
\begin{equation}\label{D:partial}
S_N f(z)=:\sum_{\abs{\alpha}_{\infty}\leq N}b_{\alpha}z^{\alpha}.
\end{equation}
Call $S_Nf$ the square partial sum of $f$.

Let  $\T^n=\{z\in \C^n:\abs{z_j}=1,j=1,\cdots,n\}$ and $U^n=\{z\in \C^n: \abs{z_j}< 1, j=1,\cdots,n\}$ denote the unit torus and polydisc, respectively.  In the sequel, quantities depending on several variables  are sometimes written in bold typeface, scalar quantities in regular, to avoid ambiguity.
For $\bm{r}=(r_1,\cdots,r_n)\in\left(\R^n\right)^+$, $\bm{\theta}=(\theta_1,\cdots,\theta_n)\in\R^n$, and $\bm{z}=(z_1, \dots z_n)\in\C^n$ let  $\bm{r}\cdot e^{i\bm{\theta}}=(r_1 e^{i\theta_1},\cdots,r_n e^{i\theta_n})$
and $\bm{r}\cdot\bm{z}=(r_1 z_1, \dots r_n z_n)$. Dilations of $\T^n$ and $U^n$ are denoted $\bm{r}\cdot\T^n=\{\bm{r}\cdot e^{i\bm{\theta}}: e^{i\bm{\theta}}\in \T^n\}$ and $\bm{r}\cdot U^n=\{\bm{r}\cdot\bm{z}:\bm{z}\in U^n\}$. 

\subsection{Hardy spaces of Reinhardt domains}\label{SS:reinhardt}
There is not a canonical definition of Hardy spaces on a general domain, especially in several variables.  See \cite{FefSte72, GunSte79, AizLif09} for a few of the definitions used. On a Reinhardt domain, the following is reasonable.

\begin{definition}\label{D:hardy}
Let $0< p<\infty$. Say $f\in H^p(\mathcal{R})$ if $f\in \mathcal{O}(\mathcal{R})$ and
$$\norm{f}_{H^p(\mathcal{R})}^p=:\sup_{\bm{r}\in \mathcal{F}}\int_{ \T^n}\abs{f(\bm{r}\cdot e^{i\bm{\theta}})}^p d\theta_1\cdots d\theta_n<\infty,$$
where $\mathcal{F}=\left\{\bm{r}:\bm{r}\cdot \T^n\subset\mathcal{R}\right\}$. 
\end{definition}

An alternate form of the integrals in Definition \ref{D:hardy} is used in section \ref{SS:ReinhardtConverge}. As shorthand, let $d\bm{\theta} = d\theta_1\cdots d\theta_n$. Then
$$\int_{ \T^n}\abs{f(\bm{r}\cdot e^{i\bm{\theta}})}^p d\bm{\theta} \approx \int_{\bm{r}\cdot \T^n}\abs{f(\xi)}d\abs{\xi_1}\cdots d\abs{\xi_n},$$
where $d\left|\xi_k\right|$ denotes arc length measure. 

\subsubsection{Density of polynomials in $H^p(\cR)$} A fundamental fact about holomorphic functions on $U$ is 
\begin{equation}\label{E:1dMonotone}
 \int_{0}^{2\pi} \left|h(r e^{i\theta})\right|^p\, d\theta \leq  \int_{0}^{2\pi} \left|h(R e^{i\theta})\right|^p\, d\theta,\qquad h\in\co(U),
\end{equation}
if $r \leq R<1$. See \cite[Theorem 1.5]{Durenbook}. 

A version of this monotonicity holds on $\cR$. For $\bm{r}=(r_1,\dots ,r_n)$ and $\bm{R}=(R_1,\dots ,R_n)$, write $\bm{r}\prec\bm{R}$ to denote $r_k\leq R_k$ for all $k=1,\dots ,n$.

\begin{lemma}\label{L:monotone} If $\bm{r}, \bm{R}\in \mathcal{F}$, $\bm{r}\prec\bm{R}$, and $f\in\co(\cR)$,
\begin{equation}\label{E:monotone}
\int_{ \T^n}\abs{f(\bm{r}\cdot e^{i\bm{\theta}})}^p d\bm{\theta}\leq \int_{ \T^n}\abs{f(\bm{R}\cdot e^{i\bm{\theta}})}^p d\bm{\theta}.
\end{equation}
\end{lemma}
\begin{proof}
\begin{align*}
\int_{ \T^n}\abs{f(\bm{r}\cdot e^{i\bm{\theta}})}^p d\bm{\theta}&= \int_{ \T^{n-1}}\left(\int_0^{2\pi}\left| f(r_1 e^{i\theta_1}, r_2 e^{i\theta_2},\dots , r_ne^{i\theta_n})\right|^p\, d\theta_1\right)d\theta_2\cdots d\theta_n \\
&\leq  \int_{ \T^{n-1}}\left(\int_0^{2\pi}\left| f(R_1 e^{i\theta_1}, r_2 e^{i\theta_2},\dots , r_ne^{i\theta_n})\right|^p\, d\theta_1\right)d\theta_2\cdots d\theta_n 
\end{align*}
by \eqref{E:1dMonotone}. Iteratively applying this to the integrals $d\theta_2\cdots d\theta_n$ gives \eqref{E:monotone}.
\end{proof}

The density of holomorphic polynomials in $H^p(U)$ is well-known, see \cite[Theorem 3.3]{Durenbook}. This fact also holds on complete Reinhardt domains in $\C^n$:

\begin{proposition}\label{P:polysDense} If $\cR\subset\C^n$ is a bounded complete Reinhardt domain and $0<p<\infty$, the set of holomorphic polynomials is dense in $H^p(\cR)$.
\end{proposition}

\begin{proof} Let $f(z)=\sum_{\alpha\in\N^n}a_{\alpha}z^{\alpha}\in H^p(\cR)$.
For $0<s <1$, define $f_s(\bm{z})= f(s\bm{z})$. 

For any $\bm{\sigma}\in\cf$, consider $I(\bm{\sigma})=: \int_{\T^n}\left|f(\bm{\sigma} e^{i\bm{\theta}}) - f_s(\bm{\sigma} e^{i\bm{\theta}})\right|^p\, d\bm{\theta}$. Lemma \ref{L:monotone} implies
\begin{align*}
I(\bm{\sigma})= \int_{\T^n}\left|f(\bm{\sigma} e^{i\bm{\theta}}) - f(s\bm{\sigma} e^{i\bm{\theta}})\right|^p\, d\bm{\theta}&\leq 2^p\bigg\{\int_{\T^n}\left|f(\bm{\sigma} e^{i\bm{\theta}})\right|^p\, d\bm{\theta}+ 
\int_{\T^n}\left|f(s\bm{\sigma} e^{i\bm{\theta}})\right|^p\, d\bm{\theta} \bigg\}\\
&\leq 2^{p+1}\, \int_{\T^n}\left|f(\bm{\sigma} e^{i\bm{\theta}})\right|^p\, d\bm{\theta}.
\end{align*}
In particular $I(\bm{\sigma})\lesssim \|f\|^p_{H^p}$.

Let $\epsilon >0$. Since $\lim_{s\to 1} \big(f(\bm{\sigma} e^{i\bm{\theta}}) - f(s\bm{\sigma} e^{i\bm{\theta}})\big)=0$ for all $\bm{\theta}$, the dominated convergence theorem gives $\rho$ such that for all
$\rho\leq s<1$
\begin{equation}\label{E:polys1}
\int_{\T^n}\left|f(\bm{\sigma} e^{i\bm{\theta}}) - f_s(\bm{\sigma} e^{i\bm{\theta}})\right|^p\, d\bm{\theta} <\epsilon.
\end{equation}
Note $\rho$ depends on both $\bm{\sigma}$ and $\epsilon$.

The function $f_{\rho}\in\co\left(\frac 1{\rho}\cR\right)$ and $\bar{\cR}\subset\subset\frac 1{\rho}\cR$. Since the power series
\begin{equation*}
f_\rho(z)=\sum_{\alpha\in\N^n}\left(a_\alpha\rho^{|\alpha|}\right) z^\alpha =:\sum_{\alpha\in\N^n} b_\alpha(\rho)\, z^\alpha
\end{equation*}
converges uniformly on $\bar\cR$, there exists $M=M(\rho,\epsilon)$ such that
\begin{equation}\label{E:polys2}
\sup_{z\in\cR}\left|f_\rho(z)-\sum_{|\alpha|_\infty\leq M}b_\alpha(\rho) z^\alpha\right| <\epsilon.
\end{equation}
Let $Q(z)=\sum_{|\alpha|_\infty\leq M}b_\alpha (\rho)z^\alpha$. Then
\begin{align*}
\int_{\T^n}\left|f(\bm{\sigma} e^{i\bm{\theta}}) - Q(\bm{\sigma} e^{i\bm{\theta}})\right|^p\, d\bm{\theta} &\leq 2^p\bigg\{I(\bm{\sigma}) +
\int_{\T^n}\left|f_\rho(\bm{\sigma} e^{i\bm{\theta}}) - Q(\bm{\sigma} e^{i\bm{\theta}})\right|^p\, d\bm{\theta} \bigg\}\\
&< 2^p(\epsilon +\epsilon^p\cdot (2\pi)^n)
\end{align*}
by \eqref{E:polys1} and \eqref{E:polys2}. This holds for any $\bm{\sigma}\in\cf$ and $\epsilon >0$ was arbitrary, so the claimed density holds.

\end{proof}

\subsection{Partial sums in $H^1(\cR)$ converge in $A^1(\cR)$}\label{SS:ReinhardtConverge} The definition of Bergman spaces is canonical: $f\in A^p(\mathcal{R})$ if
\[\norm{f}_{A^p(\mathcal{R})}^p=\int_{\mathcal{R}} \abs{f}^p dV<\infty.\]
Theorem \ref{T:main} generalizes to the pair $H^1(\mathcal{R}), A^1(\mathcal{R})$. As in one variable, the key fact is an estimate  on the operator norm of $S_N$.

\begin{proposition}\label{P:reinhardt}
There exists a constant independent of $N\in \Z^{+}$ and $f\in H^1(\mathcal{R})$ such that
$$\norm{S_N f}_{A^1(\mathcal{R})}\lesssim \norm{f}_{H^1(\mathcal{R})}\qquad\forall\, f\in H^1(\mathcal{R}),$$
where $S_Nf$ is the square partial sum \eqref{D:partial}.
\end{proposition}

\begin{proof}
Choose $\bm{r}=(r_1,\cdots,r_n)\in \mathcal{F}, \bm{R}=(R_1,\cdots,R_n)\in \mathcal{F}$ such that $r_j< R_j, j=1,\cdots,n$. For $f(z)=\sum_{\alpha}a_{\alpha}z^{\alpha}\in H^1(\mathcal{R})$, the Cauchy integral formula implies
$$a_{\alpha}=\frac{1}{(2\pi i)^n}\int_{\bm{R}\cdot \T^n} \frac{f(\xi)}{\xi_1^{\alpha_1+1} \cdots \xi_n^{\alpha_n +1}} d\xi_1\cdots d\xi_n.$$
On $\bm{r}\cdot U^n$, the $A^1$ norm of $S_Nf$ is
\begin{align*}
\int_{\bm{r}\cdot U^n} \left|\sum_{\abs{\alpha}_{\infty}\leq N}a_{\alpha}z^{\alpha}\right|& dV
=\int_{\bm{r}\cdot U^n}\abs{\sum_{\abs{\alpha}_{\infty}\leq N}\frac{1}{(2\pi i)^n}\int_{\bm{R}\cdot \T^n} \frac{f(\xi)}{\xi_1^{\alpha_1+1} \cdots \xi_n^{\alpha_n +1}} d\xi_1\cdots d\xi_n z^{\alpha} }dV(z)\\
&\lesssim\int_{\bm{r}\cdot U^n}\abs{\int_{\bm{R}\cdot \T^n}f(\xi)\sum_{\abs{\alpha}_{\infty}\leq N}\frac{z^{\alpha}}{\xi^{\alpha}} \frac{1}{\xi_1\cdots \xi_n}d\xi_1\cdots d\xi_n}dV(z)\\
&=\int_{\bm{r}\cdot U^n}\abs{\int_{\bm{R}\cdot \T^n}f(\xi)\sum_{\abs{\alpha_n}\leq N}\cdots \sum_{\abs{\alpha_1}\leq N}\frac{z_1^{\alpha_1}}{\xi_1^{\alpha_1+1}}\cdots \frac{z_n^{\alpha_n}}{\xi_1^{\alpha_n+1}}d\xi_1\cdots d\xi_n}dV(z)\\
&=\int_{\bm{r}\cdot U^n}\abs{\int_{\bm{R}\cdot \T^n}f(\xi)\prod_{j=1}^{n}\frac{1-\big(\frac{z_j}{\xi_j}\big)^{N+1}}{\xi_j-z_j}d\xi_1\cdots d\xi_n}dV(z)=II.
\end{align*}
Since $\xi\in \bm{R}\cdot \T^n$, $z\in \bm{r}\cdot U^n$, $\abs{\frac{z_j}{\xi_j}}<1$ for each $j$. Therefore
\begin{align}\label{E:4}
II&\lesssim\int_{\bm{r}\cdot U^n}\abs{\int_{\bm{R}\cdot \T^n}f(\xi)\prod_{j=1}^n\frac{1}{\xi_j-z_j}d\xi_1\cdots d\xi_n}dV(z)\notag\\
&\leq \int_{\bm{r}\cdot U^n}\int_{\bm{R}\cdot \T^n}\abs{f(\xi)}\prod_{j=1}^{n} \frac{1}{\abs{\xi_j-z_j}} d\abs{\xi_1}\cdots d\abs{\xi_n} dV(z)\notag\\
&=\int_{\bm{R}\cdot \T^n}\abs{f(\xi)}\int_{\bm{r}\cdot U^n}\prod_{j=1}^{n} \frac{1}{\abs{\xi_j-z_j}} dV(z) d\abs{\xi_1}\cdots d\abs{\xi_n}
\end{align}
by Fubini's theorem. 

Since $\mathcal{R}$ is bounded, for $\bm{r}=(r_1,\dots , r_n)\in\mathcal{F}$, the argument that gave \eqref{E:2} shows that for each $r_j$,
$\int_{r_j U} \frac{1}{\abs{\xi_j-z_j}} dV(z_j)\lesssim 1$
for a constant independent of $\xi\in\bm{R}\cdot\T^n$. Consequently
$$\int_{\bm{r}\cdot U^n} \prod_{j=1}^n\frac{1}{\abs{\xi_j-z_j}} dV(z)=\prod_{j=1}^n\int_{r_j U}\frac{1}{\abs{\xi_j-z_j}} dV(z_j)\lesssim 1.$$ 
Inserting this into \eqref{E:4} yields
\[\int_{\bm{r}\cdot U^n} \abs{\sum_{\abs{\alpha}_{\infty}\leq N}a_{\alpha}z^{\alpha}} dV(z)
\lesssim\int_{\bm{R}\cdot \T^n}\abs{f(\xi)}d\abs{\xi}_1\cdots d\abs{\xi_n}
\lesssim \norm{f}_{H^1(\mathcal{R})},
\]
with all constants independent of $N$ and $f$.

Finally note that for any $f\in A^1(\mathcal{R})$ there are sequences $\{\bm{r}_k\}_{k=1}^{\infty}\subseteq \mathcal{F}$ such that  $\norm{f}_{A^1(\mathcal{R})}=\lim_{k\rightarrow \infty}\int_{\bm{r}_k\cdot U^n}\abs{f}dV(z)$, by the dominated convergence theorem. Therefore
\begin{equation*}\label{E:lastBound}
\norm{S_N f}_{A^1(\mathcal{R})}\lesssim \norm{f}_{H^1(\mathcal{R})}
\end{equation*}
as claimed.
\end{proof}

Convergence of $S_Nf$ is $A^1(\cR)$ now follows as in the conclusion to the proof of Theorem \ref{T:main}. Let $T_N=S_N -\text{id}$, $X=H^1(\cR)$, $Y=A^1(\cR)$, and $D=\{\text{holomorphic polynomials}\}$. 
Note $T_N p\equiv 0$ for any $p\in D$ if $N$ is large enough. The hypotheses of Lemma \ref{L:ZhuPlus} are thus satisfied. 
Proposition \ref{P:reinhardt}  and this lemma yield

\begin{corollary} If $f\in H^1(\cR)$ and $S_Nf$ is given by \eqref{D:partial}, then
$$\left\|S_Nf-f\right\|_{A^1(\cR)}\longrightarrow 0\qquad\text{as } N\to\infty.$$
\end{corollary} 
\medskip

%%%%%%%%%%%%%%%%%%%%%%%%%%%%%%%%%%%%%%%%%%%%%%%%%%%%%%%%%%%%%%%%%%%%%%%

\bibliographystyle{acm}
\bibliography{McnXio18}

\begin{thebibliography}{10}

\bibitem{AizLif09}
{\sc Aizenberg, L., and Liflyand, E.}
\newblock Hardy spaces in {R}einhardt domains and {H}ausdorff operators.
\newblock {\em Illinois J. Math. 53}, 4 (2009), 1033--1049.

\bibitem{SiskAle95}
{\sc Aleman, A., and Siskakis, A.~G.}
\newblock An integral operator on {$H^p$}.
\newblock {\em Complex Variables Theory Appl. 28}, 2 (1995), 149--158.

\bibitem{Durenbook}
{\sc Duren, P.}
\newblock {\em $H^p$ spaces}, vol.~38 of {\em Pure and Applied Mathematics (New
  York)}.
\newblock Academic Press, New York-London, 1970.

\bibitem{DurSch04}
{\sc Duren, P., and Schuster, A.}
\newblock {\em Bergman spaces}, vol.~Mathematical surveys and mongraphs.
\newblock American Mathematical Society, 2004.

\bibitem{EdhMcN16}
{\sc Edholm, L.~D., and McNeal, J.~D.}
\newblock The {B}ergman projection on fat {H}artogs triangles: ${L}^p$
  boundedness.
\newblock {\em Proc. Amer. Math. Soc. 144}, 5 (2016), 2185--2196.

\bibitem{EdhMcN16b}
{\sc Edholm, L.~D., and McNeal, J.~D.}
\newblock Bergman subspaces and subkernels: degenerate ${L}^p$ mapping and
  zeroes.
\newblock {\em J. Geom. Anal. 27}, 4 (2017), 2658--2683.

\bibitem{FefSte72}
{\sc Fefferman, C., and Stein, E.~M.}
\newblock ${H}^p$ spaces in several variables.
\newblock {\em Acta Math. 129}, 3-4 (1972), 137--193.

\bibitem{ForRud74}
{\sc Forelli, F., and Rudin, W.}
\newblock Projections on spaces of holomorphic functions in balls.
\newblock {\em Ind. Univ. Math. J. 24\/} (1974), 593--602.

\bibitem{Sisk01}
{\sc Galanopoulos, P., and Siskakis, A.~G.}
\newblock Hausdorff matrices and composition operators.
\newblock {\em Illinois J. Math. 45}, 3 (2001), 757--773.

\bibitem{Garnettbook}
{\sc Garnett, J.}
\newblock {\em Bounded analytic functions}.
\newblock Academic Press, New York, 1981.

\bibitem{GunSte79}
{\sc Gundy, R., and Stein, E.~M.}
\newblock ${H}^p$ theory for the poly-disc.
\newblock {\em Proc. Nat. Acad. Sci. USA 76}, 3 (1979), 1026--1029.

\bibitem{Lif13}
{\sc Liflyand, E.}
\newblock Hausdorff operators on {H}ardy spaces.
\newblock {\em Eurasian Math. J. 4}, 4 (2013), 101--141.

\bibitem{PavNow10}
{\sc Nowak, M., and Pavlovi\'c, M.}
\newblock On the {L}ibera operator.
\newblock {\em J. Math. Anal. Appl. 370}, 2 (2010), 588--599.

\bibitem{PapVir08}
{\sc Papadimitrakis, M., and Virtanen, J.~A.}
\newblock Hankel and {T}oeplitz transforms on {$H^1$}: continuity, compactness
  and {F}redholm properties.
\newblock {\em Integral Equations Operator Theory 61}, 4 (2008), 573--591.

\bibitem{RudinFunctiontheory}
{\sc Rudin, W.}
\newblock {\em Function theory in the unit ball in {$\C^n$}}, vol.~241 of {\em
  Grundlehren der Mathematischen Wissenschaften}.
\newblock Springer-Verlag, New York, 1980.

\bibitem{Sisk87}
{\sc Siskakis, A.~G.}
\newblock Composition semigroups and the {C}es\`aro operator on {$H^p$}.
\newblock {\em J. London Math. Soc. (2) 36}, 1 (1987), 153--164.

\bibitem{Sisk90}
{\sc Siskakis, A.~G.}
\newblock The {C}es\`aro operator is bounded on {$H^1$}.
\newblock {\em Proc. Amer. Math. Soc. 110}, 2 (1990), 461--462.

\bibitem{Stegenga76}
{\sc Stegenga, D.~A.}
\newblock Bounded {T}oeplitz operators on {$H^{1}$} and applications of the
  duality between {$H^{1}$} and the functions of bounded mean oscillation.
\newblock {\em Amer. J. Math. 98}, 3 (1976), 573--589.

\bibitem{Zhu91}
{\sc Zhu, K.}
\newblock Duality of {B}loch spaces and norm convergence of {T}aylor series.
\newblock {\em Mich. Math. J. 38\/} (1991), 89--101.

\bibitem{ZhuBergmanbook}
{\sc Zhu, K.}
\newblock {\em Spaces of holomorphic functions in the unit ball}, vol.~226 of
  {\em Graduate Texts in Mathematics}.
\newblock Springer-Verlag, New York, 2005.

\end{thebibliography}

\end{document}